\documentclass[12pt,a4paper]{amsart}
\usepackage[utf8]{inputenc}
\usepackage[T1]{fontenc}
\usepackage{amsmath}
\usepackage{amssymb}
\usepackage{amsfonts}
\usepackage{amsthm}
\usepackage[all]{xy}
\usepackage{tikz}
\usepackage{hyperref}
\usepackage[shortlabels]{enumitem}
\usepackage{bbold}
\usepackage{geometry}
\usepackage{tabularray}
\usepackage{geometry}
\usepackage{apptools}
\usepackage{xcolor}
\usepackage{faktor}
\usepackage{xfrac}   

\newtheorem{thm}{Theorem}

\theoremstyle{definition}
\newtheorem{example}{Example}
\newtheorem{remark}{Remark}

\def\Og{\operatorname{OG}}
\def\Oo{\operatorname{O}}
\def\SO{\operatorname{SO}}
\def\Sym{\operatorname{Sym}}
\def\PP{\operatorname{P}}
\def\BB{\operatorname{B}}
\def\Pr{\mathbb{P}}
\def\coh{\operatorname{H}}
\def\RR{\mathcal{R}}
\def\CC{\mathbb{C}}
\def\QQ{\mathbb{Q}}
\def\zz{\mathbf{z}}
\def\tt{\mathbf{t}}
\def\TT{\operatorname{T}}
\def\pt{\operatorname{pt}}
\def\id{\operatorname{id}}
\def\ff{\varphi}
\def\rk{{\rm rk}}
\newcommand{\spa}{span}
\newcommand{\OG}[2]{\Og(#1,#2)}
\DeclareMathOperator{\res}{Res}
\usepackage[nice]{nicefrac}
\def\sfrac#1#2{\text{\large$\nicefrac{#1}{#2}$}}

\begin{document}

\author{Andrzej Weber}\address{Institute of Mathematics, University of Warsaw, Poland}
\email{aweber@mimuw.edu.pl}
\thanks{A.W.~supported by Polish National Science Center grant number 2022/47/B/ST1/01896}

\author{Magdalena Zielenkiewicz}
\address{Institute of Mathematics, University of Warsaw, Poland}
\email{magdaz@mimuw.edu.pl}

\title[Push-forward formulas for even orthogonal Grassmannians]{A note on the push-forward formulas\\ for even orthogonal Grassmannians }
\date{}

\begin{abstract}
We revisit residue formulas for the push-forward in the cohomology of the even orthogonal Grassmannian. This space has two components, and the formula for a single component demands separate attention. We correct errors spread throughout the literature.
\end{abstract}

\maketitle

Push--forward formulas for Grassmann bundles play a significant role in intersection theory. We take a closer look at the residue formulas applicable to bundles having even orthogonal Grassmannians as their fibers. While the existing literature covers results only for fibers
$\OG{n}{2n}=\sfrac{\Oo_{2n}}{\PP}$ -- a variety with two components, statements regarding $\OG{n}{2n}^+=\sfrac{\SO_{2n}}{\PP}$ remain unproven and are formulated incorrectly. We present a complete solution with a proof in the language of equivariant cohomology. Our formulas in equivariant cohomology are equivalent to the corresponding expressions given in terms of the formal Chern roots of the involved vector bundle.

\section{The main results}
Let $V$ be a $2n$-dimensional complex vector space equipped with a non-degenerate, symmetric bilinear form $\Omega$. 
Let $\OG{n}{2n}$ be the variety, called orthogonal Grassmannian, parametrizing 
 maximal isotropic subspaces of $V$. 
 It can be written as the quotient $\sfrac{\Oo(V)}{\PP}$, where $\PP$ is the parabolic subgroup stabilising a fixed maximal isotropic subspace.
The variety $\OG{n}{2n}$
has two isomorphic connected components,  which we denote $\OG{n}{2n}^{+}$ and $\OG{n}{2n}^{-}$. 
The sign depends on the choice of a reference isotropic subspace, which will be made canonical after choosing a canonical form of $\Omega$. 
We study the push-forward along the projection to the one-point space $\pi: \OG{n}{2n}^{\pm} \to \pt$ in torus-equivariant cohomology. There are residue formulas for such push-forwards (see for example \cite{dp,mz} or \cite{wz} for a K--theoretical version), but they consider the two connected components together, i.e. the push-forward along the projection $\sfrac{\Oo(V)}{\PP} \to \pt$.
While there is an isomorphism $\OG{n}{2n}^{+} \simeq \OG{n-1}{2n-1}$ (see e.g.~\cite[p~29.]{tevelev}), it does not easily yield compact push-forward formulas for $\OG{n}{2n}^{+}$. The main obstacle is  the rank of the maximal torus in $\SO_{2n-1}$ which is smaller than the rank of the maximal torus in $\SO(V)\simeq\SO_{2n}$. The action of the additional $\CC^*$ is not visible in the presentation  $\OG{n}{2n}^{+}\simeq\sfrac{\SO_{2n-1}}{(\PP\cap \SO_{2n-1})}$. \\

Let $\TT\simeq(\CC^*)^n$ be a torus. 
For a $\TT$-space $X$ let $\coh_{\TT}^{*}(X)$ denote the $\TT$-equivariant cohomology ring of $X$ with coefficients in $\QQ$. For $X = \OG{n}{2n}^{\pm}$ and the action of a maximal torus $\TT\subset\SO(V)$, the fixed point set of the action is finite. Any class $\alpha  \in \coh_{\TT}^{*}(X)$ is determined by its restrictions to fixed points. Since $\coh_{\TT}^{*}(\pt) \simeq \QQ[t_1, \dots, t_n]$, each such restriction is a polynomial in the characters $t_1,\dots , t_n$ of $\TT\simeq (\CC^{*})^{n}$.  We will assume that $\alpha$ comes from $\coh_{\SO_{2n}}^{*}(\OG{n}{2n}^\pm)$, then the restriction of $\alpha$ at one fixed point determines all other restrictions. They differ by the action of the Weyl group.
\\

We will be interested in the classes of the form $\alpha=\ff(\RR^\vee)$, where $\RR$ is the tautological bundle of rank $n$ over $\OG{n}{2n}^\pm$ and $\ff$ is a characteristic class, that is, a polynomial in the Chern classes, which is identified with a symmetric polynomial in $n$ variables. The Schur classes are of special interest. 
We prefer to look at the dual bundle $\RR^\vee$ to get rid of 
annoying signs and to have formulas compatible with \cite{PraRat}. Also note that the tangent space $T_{\OG{n}{2n}}$ is isomorphic to $\wedge^2(\RR^\vee)$, which will make our local contributions to the push-forward entirely expressed in terms of the Chern roots of $\RR^\vee$. 
In this note, we show that the push-forward of characteristic classes of $\RR^\vee$ can be expressed as residues of certain meromorphic functions.\\

Let us choose a basis $\{ e_1, \dots, e_n, f_n, \dots, f_1 \}$ of $V$ in such a way that 
$\Omega(e_i, f_i)=1$ and $\Omega(e_i, e_j) = 0 = \Omega(f_i, f_j) = \Omega(e_i, f_j)$ for $i \neq j$, then the fixed points are the subspaces
\[p=\spa \{ w_1, \dots, w_n \}, \textrm{ where } w_i=e_i \textrm{ or } w_i = f_i.\]
A fixed point is contained in the component $\OG{n}{2n}^{+}$ if and only if it has an even number of vectors $f_i$ in the presentation above. 
Thus the parabolic group $\PP$ is block-lower-triangular. The maximal torus $\TT$ is the standard diagonal torus. Let $t_1,t_2,\dots,t_n$ be the characters corresponding to the first $n$ coordinates. In other words, the torus consists of the diagonal matrices $$x=diag(x_1,\dots,x_n,x_n^{-1},\dots,x_1^{-1})$$ and $t_i(x)=x_i$.
We introduce formal variables $z_1, \dots, z_n$ and given a meromorphic form $$g(z_1,\dots,z_n)dz_1 \dots dz_n$$ we denote by $\res_{\zz = \infty}g(\zz) d\zz$ the iterated residue operation
\[\res_{z_n=\infty}(\res_{z_{n-1}=\infty} (\dots (\res_{z_1=\infty} g(z_1,\dots, z_n) dz_1)\dots )dz_{n-1}) dz_n. \]
We prove the following:

\begin{thm}\label{thm:push}
Let $\alpha =\ff(\RR^\vee) \in \coh_{\TT}^{*}(\OG{n}{2n}^{+})$ and let $\pi^{+}: \OG{n}{2n}^{+} \to \pt$ be the constant map. Then 
\[\pi^{+}_{*} \alpha = \frac{2^{n-1}}{n! \prod_{i} t_i} \res_{\zz =\infty} \ff(\zz)\frac{ \prod_{i\neq j}(z_j-z_i) \prod_{i < j}(z_i + z_j)\prod_{i} z_i \  (\prod_{i} t_i + \prod_{i}z_i)}{\prod_{i,j}(t^2_i-z^2_j)}d\zz.\]

Let  $\alpha =\ff(\RR^\vee) \in \coh_{\TT}^{*}(\OG{n}{2n}^{-})$ and let $\pi^{-}: \OG{n}{2n}^{-} \to \pt$. Then 
\[\pi^{-}_{*} \ff(\RR^\vee) = \frac{2^{n-1}}{n! \prod_{i} t_i} \res_{\zz =\infty} \ff(\zz)\frac{ \prod_{i\neq j}(z_j-z_i) \prod_{i < j}(z_i + z_j)\prod_{i} z_i \  (\prod_{i} t_i - \prod_{i}z_i)}{\prod_{i,j}(t_i^2-z_j^2)}d\zz.\]

\end{thm}

We give an alternative form of residue formula, which is slightly less complicated:
\begin{thm}\label{cor:push}
With the assumptions as above
\[\pi^{+}_{*} \alpha = \frac{2^{n-1}}{\prod_{i} t_i\;\prod_{i<j}(t_j^2 - t_i^2)} \res_{\zz =\infty} \ff(\zz)\frac{ \prod_{i<j}(z_j-z_i) \prod_{i} z_i \ (\prod_{i} t_i + \prod_{i}z_i)}{\prod_{i}(t_i^2-z_i^2)} d\zz.\]
\[\pi^{-}_{*} \alpha  = \frac{2^{n-1}}{\prod_{i} t_i\;\prod_{i<j}(t_j^2 - t_i^2)} \res_{\zz =\infty}\ff(\zz) \frac{ \prod_{i<j}(z_j-z_i) \prod_{i} z_i \  (\prod_{i} t_i - \prod_{i}z_i)}{\prod_{i}(t_i^2-z_i^2) }d\zz.\]
\end{thm}

Using the push-forward formulas of Theorem \ref{cor:push}, we compare the push-forwards of the Schur polynomials on the two components. The result is the following.

\begin{thm}\label{thm:schur} Let $\lambda=(\lambda_1, \dots , \lambda_n)$ be a partition, and $s_{\lambda}$ the corresponding Schur polynomial. Denote by $\rho(n)$ the partition $\rho(n) = (n, n-1, \dots, 1)$.
Then the following holds:
\[\pi^{+}_{*} s_{\lambda}(\RR^\vee) = 2^{n-1}\begin{cases}
  s_{\mu}(t_1^2, \dots, t_n^2)  &  \lambda = 2 \mu + \rho(n-1) \\
  t_1\dots t_n s_{\mu}(t_1^2, \dots, t_n^2)  &  \lambda = 2 \mu + \rho(n)\\
  0 &  \text{otherwise,}
\end{cases}
\]

\[\pi^{-}_{*} s_{\lambda}(\RR^\vee) = 2^{n-1} \begin{cases}
  s_{\mu}(t_1^2, \dots, t_n^2)  &  \lambda = 2 \mu + \rho(n-1) \\
 - t_1\dots t_n s_{\mu}(t_1^2, \dots, t_n^2)  &  \lambda = 2 \mu + \rho(n)\\
  0 &  \text{otherwise.}
\end{cases}
\]
In particular, the push-forwards of the Schur class $s_{\lambda}(\RR)$ on $\OG{n}{2n}^{+}$ and $\OG{n}{2n}^{-}$ are equal, unless $\lambda$ is of the form $2 \mu + \rho(n)$ for some partition $\mu$, in which case the push-forwards differ by the factor $(-1)$.
\end{thm}
 
 The case  $\lambda = 2 \mu + \rho(n)$ is  overlooked in \cite[Theorem 5.21]{PraRat}.
Note that in \cite[below Theorem 3.1]{dp} it is claimed that the integrals over $\OG{n}{2n}^{+}$ and $\OG{n}{2n}^{-}$ are equal. This is a false statement.

\begin{example}Let $n=2$ and
$$\ff(z_1,z_2)=s_{2,1}(z_1,z_2)=z_1 z_2 (z_1 + z_2)\,.$$ 
The orthogonal Grassmannian is of dimension one and 
$$\OG{1}{2}^\pm\simeq\Pr^1\,.$$ The tangent bundle $T_{\OG{1}{2}}$ is isomorphic to $\wedge^2\RR^\vee\,.$
Then, applying the localization formula for the push-forward (see \eqref{locformula} below), we obtain
$$\pi^+_*(\alpha)=\frac{\ff(t_1, t_2)}{t_1+ t_2} + \frac{\ff(-t_1,- t_2)}{-t_1 - t_2}=\frac{t_1 t_2(t_1+ t_2)}{t_1+ t_2} + \frac{t_1 t_2 (-t_1 - t_2)}{-t_1 - t_2}=2 t_1 t_2\,,$$
while 
$$\pi^-_*(\alpha)=\frac{\ff(-t_1, t_2)}{-t_1 + t_2} + \frac{\ff(t_1,- t_2)}{t_1 - t_2}=\frac{-t_1 t_2(-t_1+ t_2)}{-t_1+ t_2} + \frac{-t_1 t_2 (t_1- t_2)}{t_1 - t_2}=-2 t_1 t_2\,.$$
Of course the extreme case $n=1$, $\ff(z_1)=s_1(z_1)$ is  more appealing. We leave the contemplation of that case to the reader.
\end{example}

\section{Cohomological formulas for the push-forward}

The proof of Theorem \ref{thm:push} is based on the following push-forward formula of \cite[Formula 4]{mz0}. 
Let $\alpha=\ff(\RR^\vee) \in \coh_{\TT}^{*}(\OG{n}{2n})$ and let $\pi: \OG{n}{2n} \to \pt$. Then
\begin{equation}\pi_{*} \alpha = \frac{2^{n}}{n!} \res_{\zz =\infty} \frac{\ff(\zz) \prod_{i\neq j}(z_j-z_i) \prod_{i < j}(z_i + z_j)\prod_{i} z_i }{\prod_{i,j}(z_i+t_j)(t_i-z_j)}d\zz. \label{twocomponents}
\end{equation}

\begin{remark} The above formula can be rewritten using the formalism of Chern roots:  
\begin{equation}\pi_{*} \alpha = \frac{1}{n!} \res_{\zz =\infty} \frac{\ff(\zz)\; \Delta(\zz) \;eu(\Sym^2(\zz))}{eu(\zz\otimes \tt)\;eu(\tt\otimes\zz^\vee))}d\zz\,,\end{equation}
Here the set of variables $\zz$ is identified with the formal vector space spanned by the characters belonging to $\zz$, the Euler class is the product of its elements and $$\Delta(\zz)=  \prod_{i\neq j}(z_j-z_i)\,.$$
The conceptual proof of the analogous formula in K--theory can be found  in \cite[\S5.2]{wz}.
\end{remark}

Let us briefly sketch the proof of Formula \ref{twocomponents}. By the Atiyah--Bott--Berline--Vergne localization formula (\cite{ab,bv}), the push-forward $\pi_{*} \alpha$ can be computed by summing up the contributions from all the fixed points, 
\begin{equation}\label{locformula}\pi_{*} \alpha = \sum_{p \in X^{\TT}} \frac{\alpha_{|p}}{eu(p)},\end{equation}
where $eu(p)$ is the product of the characters of $\TT$ which appear in the tangent representation at the fixed point $p$. On the other hand, taking the residue at infinity gives minus the sum of residues at all the poles of the function under the residue. Iterating the procedure gives a sum of residues at $z_1 = \pm t_1, \dots, z_n=\pm t_n$ as all other residues vanish, and each such residue corresponds to exactly one fixed point contribution in the localization formula.\\

There is a more  economical (from the computational point of view) expression which can serve as a factor under the residue.
The complicated product
$$\frac{2^n}{n!}\;\frac{\prod_{i\neq j}(z_j-z_i) \prod_{i < j}(z_i + z_j)\prod_{i} z_i}{\prod_{i,j}(z_i+t_j)(t_i-z_j)}$$
can be replaced by
$$\frac{2^{n}}{ \prod_{i<j}(t_j^2 - t_i^2)} \;\frac{\prod_{i<j}(z_j-z_i) \prod_{i} z_i}{\prod_{i}(t_i^2-z_i^2)}\,.$$
It is enough to check that it has the right residue at 
$$(z_1,z_2,\dots,z_n)=(\pm t_1,\pm t_2,\dots,\pm t_n)\,,$$ namely
$$\frac1{eu(\wedge^2( \pm\tt))}=\frac1{\prod_{i<j}(\pm t_i\pm t_j)}\,.$$
\begin{example}Let $n=2$ and $$K(z_1,z_2)=\frac{(z_2-z_1) z_1 z_2}{(t_1^2-z_1^2)(t_2^2-z_2^2)}\,.$$
We have 
\begin{align*}&\res_{z_1=t_1}\res_{z_2=t_2}K(z_1,z_2)d\zz= \frac{-t_1+t_2}4\,,\\
&\res_{z_1=t_1}\res_{z_2=-t_2}K(z_1,z_2)d\zz= \frac{-t_1-t_2}4\,,\\
&\res_{z_1=-t_1}\res_{z_2=t_2}K(z_1,z_2)d\zz= \frac{t_1+t_2}4\,,\end{align*}
etc. Correcting by the constant factor $\frac{4}{ t_2^2 - t_1^2} $ we obtain what we need.

\end{example}
\begin{thm}\label{shortrez}
Let $\alpha=\ff(\RR^\vee) \in \coh_{\TT}^{*}(\OG{n}{2n})$. Then 
\[\pi_* \alpha = \frac{2^{n}}{ \prod_{i<j}(t_j^2 - t_i^2)}\; \res_{\zz =\infty} \frac{\ff(\zz) \prod_{i<j}(z_j-z_i) \prod_{i} z_i}{\prod_{i}(t_i^2-z_i^2)}d\zz.\]
\end{thm}
This formula was already applied in \cite{mz1} to compute the push-forward of Schur classes.

\begin{remark}Yet another formula is provided by \cite[Theorem 3.1]{dp}. It applies to any partial flag variety. For the orthogonal Grassmannian it specializes to  
$$\pi_*(\ff(\RR^\vee))=\Big(\ff(\zz)\;\prod_{i=1}^n(2z_i)\,\prod_{i<j}((z_i-z_j)(z_i+z_j))\,\prod_{i=1}^n s_{1/z_i}(\CC^{2n})\Big)_{[z_1^{2n-1}z_2^{2n-2}\dots z_n^{n}]}.$$
The expression above is written in formal variables $z_i$, while the $\pm t_i$ serve as the Chern roots of $\CC^{2n}$.  
Here $f(\zz)_{[\zz^\lambda]}$ denotes the coefficient of $\zz^\lambda$ in the expansion at infinity. The Segre formal series is equal to
$$s_{1/z_i}(\CC^{2n})=\prod_{j=1}^n \frac1{(1-\tfrac{ t_j}{z_i})(1+\tfrac{ t_j}{z_i})}\,,\qquad\text{expansion at }z_i=\infty\,.$$ 
In terms of residues the formula of \cite{dp} can be written as
\begin{align*}\pi_*(\ff(\RR^\vee))&=(-1)^n\res_{\zz=\infty}\ff(\zz)\frac{\prod_{i=1}^n(2z_i)\;\prod_{i<j}((z_i-z_j)(z_i+z_j))}
{\prod_{i=1}^n z_i^{2n-i+1}\prod_{i,j}(1-\tfrac{ t_j}{z_i})(1+\tfrac{ t_j}{z_i})}d\zz\\
&=(-1)^n \res_{\zz=\infty}\ff(\zz)\frac{\prod_{i=1}^n z_i^{i-1}\;\prod_{i=1}^n (2z_i)\;\prod_{i<j}(z_i^2-z_j^2)}
{\prod_{i,j}(z_j^2- t_i^2)}d\zz\\
&=2^n \res_{\zz=\infty}\ff(\zz)\frac{\prod_{i=1}^n z_i^{i}\;\prod_{i<j}(z_i^2-z_j^2)}
{\prod_{i,j}(t_i^2- z_j^2)}d\zz
\,.
\end{align*}
This  formula can be obtained from \eqref{twocomponents}, altough it has less factors.
\end{remark}

\section{Proof of formulas for the components $\OG{n}{2n}^\pm$}

In the Atiyah-Bott-Berline-Vergne formula for $\OG{n}{2n}^{\pm}$ the contributions from the fixed points remain the same for those fixed points which are in the component in question (and are zero for the remaining points). 
The fixed points of the torus action on $\OG{n}{2n}\simeq\sfrac{\Oo(V)}{\PP}$ are the isotropic coordinate subspaces.
\\

\begin{proof}[Proof of Theorem \ref{thm:push}]
In the localization formula \eqref{locformula}, the contribution form $p = \spa \{ w_1, \dots, w_n \}$ is equal to 
\[ \frac{\ff(\pm t_1, \dots, \pm t_n)}{\prod_{i,j}(\pm t_i \pm t_j)},\]
where the plus signs appear at those indices $i$ for which $w_i = f_i$. 
With our convention, after identifying $\OG{n}{2n}\simeq \sfrac{\Oo_{2n}}{\PP}$ the contribution at $\id \PP$ is equal to
\[ \frac{\ff( t_1, \dots, t_n)}{\prod_{i,j}( t_i + t_j)}\,.\]
\\

Our formula for the push-forward for $\OG{n}{2n}^{+}$ given in Theorem \ref{thm:push} differs from the formula \eqref{twocomponents} by the factor 
\[X^{+}(z_1,\dots,z_n) :=\frac{\prod_{i} t_i + \prod_{i}z_i}{2 \prod_{i}t_i}.\]

This factor does not introduce new poles in the form under the residue, as the residue is taken with respect to the $z_i$'s. The only non-vanishing residues come from $z_i = \pm t_i$, and they are equal to the fixed-point contributions in the Atiyah-Bott-Berline-Vergne formula, multiplied by the factor $X$ evaluated at $z_i = \pm t_i$. We have

\[X^{+}(\pm t_1,\dots, \pm t_n) = \frac{\prod_{i} \pm t_i + \prod_{i}t_i}{2 \prod_{i}t_i} = \frac{((-1)^{k} + 1)\prod_i t_i}{2 \prod_i t_i},\]
where $k$ is the number of indices for which $z_i = -t_i$. This expression vanishes for $k$ odd and is equal to 1 otherwise. Therefore it corrects the push-forward formula \eqref{twocomponents} by multiplying by zero the contributions coming from the fixed points in the component $\OG{n}{2n}^{-}$ and leaving the remaining contributions intact. \\

Similarly, for $\OG{n}{2n}^{-}$, the factor 
\[X^{-}(z_1,\dots,z_n) :=\frac{\prod_{i} t_i - \prod_{i} z_i}{2 \prod_{i}t_i}\]
vanishes at fixed points belonging to $\OG{n}{2n}^{+}$ and is equal to 1 on fixed points in $\OG{n}{2n}^{-}$.
\end{proof}

The identical argument shows that the push-forward formulas derived from Theorem \ref{shortrez} also work. We obtain Theorem \ref{cor:push}.

\section{Push-forwards of Schur polynomials}

We apply Theorem \ref{cor:push} to prove Theorem \ref{thm:schur}. A key trick in the computation is to use the relation between the residue at infinity and the residue at zero, $\res_{z=\infty} f(z) dz =\res_{z_i=0} \frac{(-1)}{z^2}f(z^{-1}) dz$, so that
\[\res_{\zz = \infty} f(\zz)d\zz = \res_{\zz=(0,\dots,0)} \frac{(-1)^n}{z_1^2\dots z_n^2} f(\zz^{-1})d\zz,\]
and the latter residue is simply the coefficient at $z_1^{-1}\dots z_n^{-1}$ in the series expansion for the expression under the residue.

\begin{proof}[Proof of Theorem \ref{thm:schur}]

One of the ways of defining Schur polynomials is the following formula (see e.g.~\cite[\S I.3]{MacD}):

\[s_{\lambda_1,\dots,\lambda_n}(z_1, \dots, z_n):=\frac{\det [z_j^{\lambda_i + n - i}]_{i, j=1, \dots,n} }{\prod_{1 \leq i < j \leq n}(z_i-z_j)},\]
where 
\[[z_j^{\lambda_i + n - i}]_{i, j=1, \dots,n} = 
\left[\begin{array}{c c c c}
z_1^{\lambda_1 +n-1} & z_2^{\lambda_1 +n-1} & \dots &  z_n^{\lambda_1 +n-1}\\ 
z_1^{\lambda_2 +n-2} & z_2^{\lambda_2 +n-2} & \dots &  z_n^{\lambda_2 +n-2}\\ 
\vdots & \vdots & \ddots & \vdots \\ 
z_1^{\lambda_n} & z_2^{\lambda_n} & \dots & z_n^{\lambda_n} \end{array}\right]. \]

Using the push-forward formula of Corollary \ref{cor:push} we get:

\begin{align*} 
\pi_{*}^{+} s_\lambda&(\RR^\vee) = \frac{2^{n-1}}{\prod_{i<j}(t_j^2 - t_i^2)\prod_{i} t_i} \res_{\zz =\infty} s_{\lambda}(\zz)\frac{ \prod_{i<j}(z_j-z_i) \prod_{i} z_i \ (\prod_{i} t_i + \prod_{i}z_i)}{\prod_{i}(t_i^2-z_i^2) }d\zz \\
&= \frac{2^{n-1}}{\prod_{i<j}(t_j^2 - t_i^2)}\res_{\zz =\infty}\frac{\det [z_j^{\lambda_i + n - i}]_{i, j}}{\prod_{i < j}(z_i-z_j)}\cdot \frac{ \prod_{i<j}(z_j-z_i) \prod_{i} z_i \ (\prod_{i} z_i + \prod_{i}t_i)}{ \prod_{i}(t_i^2-z_i^2)\prod_{i} t_i}d\zz \\
&= \frac{2^{n-1}(-1)^{n \choose 2}}{\prod_{i<j}(t_j^2 - t_i^2) } \res_{\zz =\infty} \frac{\det [z_j^{\lambda_i + n - i}]_{i, j} \prod_{i} z_i \ (\prod_{i} z_i + \prod_{i}t_i)}{ \prod_{i}(t_i^2-z_i^2) \prod_{i} t_i}d\zz, 
\end{align*} 
where the sign $(-1)^{n \choose 2}$ comes from cancelling out the products of $(z_i - z_j)$ in the numerator and the denominator. Denoting the constant factor in front of the residue by $$C := \frac{2^{n-1}(-1)^{n \choose 2}}{\prod_{i<j}(t_j^2 - t_i^2) }=\frac{2^{n-1}}{\prod_{i<j}(t_i^2 - t_j^2) }$$ and changing the residue at infinity to a residue at zero we get:

\begin{align*} 
\pi_{*}^{+} s_\lambda(\RR^\vee) &= C\cdot \res_{\zz =0} \frac{(-1)^n}{z_1^2 \dots z_n^2} \frac{\det [z_j^{-\lambda_i - n + i}]_{i, j} \prod_{i} z_i^{-1} \ (\prod_{i} z_i^{-1} + \prod_{i}t_i)}{ \prod_{i}(t_i^2-z_i^{-2})  \prod_{i} t_i}d\zz \\
&= C \cdot \res_{\zz =0} \frac{\det [z_j^{-\lambda_i - n + i}]_{i, j} \prod_{i} z_i^{-2} \ (1 + \prod_{i}t_i z_i)}{ \prod_{i}(1-t_i^2 z_i^2)  \prod_{i} t_i}d\zz.
\end{align*} 

Splitting the expression under the residue as the sum of two parts one gets

\[
 \frac{\det [z_j^{-\lambda_i - n + i}]_{i, j} \prod_{i} z_i^{-2} \ (1 + \prod_{i}t_i z_i)}{ \prod_{i}(1-t_i^2 z_i^2)  \prod_{i} t_i}  \]
\[=\frac{\det [z_j^{-\lambda_i - n + i}]_{i, j} \prod_{i} z_i^{-2}}{ \prod_{i}(1-t_i^2 z_i^2)  \prod_{i} t_i} + \frac{\det [z_j^{-\lambda_i - n + i}]_{i, j} \prod_{i} z_i^{-2} \ (\prod_{i}t_i z_i)}{ \prod_{i}(1-t_i^2 z_i^2)  \prod_{i} t_i}.\]

As taking the residue is an additive operation, we deal with the  two summands separately. Expanding the determinant using the permutation formula and expanding $\frac{1}{1-t_i^2z_i^2}$ into a power series we have

\begin{align*}
 A &= \frac{\det [z_j^{-\lambda_i - n + i}]_{i, j} \prod_{i} z_i^{-2}}{ \prod_{i}(1-t_i^2 z_i^2)  \prod_{i} t_i} \\
 &= \prod_{i} t_i^{-1} z_i^{-2} \sum_{\sigma \in \Sigma_n} (-1)^{sgn(\sigma)} z_{\sigma(1)}^{-\lambda_1 - n + 1} \dots z_{\sigma_n}^{-\lambda_n} \sum_{i_1, \dots, i_n=0}^{\infty} (t_1 z_1)^{2i_1} \dots (t_n z_n)^{2i_n}  \\
 &= \sum_{\sigma \in \Sigma_n} (-1)^{sgn(\sigma)} \sum_{i_{\sigma(1)}, \dots, i_{\sigma(n)}=0}^{\infty} t_{i_{\sigma(1)}}^{2 i_{\sigma(1)}-1}\dots t_{i_{\sigma(n)}}^{2 i_{\sigma(n)}-1} z_{i_{\sigma(1)}}^{2 i_{\sigma(1)}-2-(\lambda_1 + n - 1)} \dots z_{i_{\sigma(n)}}^{2 i_{\sigma(n)}-2-\lambda_n},
\end{align*}
and similarly for the other summand
\begin{align*}
 B &= \frac{\det [z_j^{-\lambda_i - n + i}]_{i, j} \prod_{i} z_i^{-2} \prod_i t_i z_i}{ \prod_{i}(1-t_i^2 z_i^2)  \prod_{i} t_i} \\
 &= \sum_{\sigma \in \Sigma_n} (-1)^{sgn(\sigma)} \sum_{i_{\sigma(1)}, \dots, i_{\sigma(n)}=0}^{\infty} t_{i_{\sigma(1)}}^{2 i_{\sigma(1)}}\dots t_{i_{\sigma(n)}}^{2 i_{\sigma(n)}} z_{i_{\sigma(1)}}^{2 i_{\sigma(1)}-1-(\lambda_1 + n - 1)} \dots z_{i_{\sigma(n)}}^{2 i_{\sigma(n)}-1-\lambda_n}.
\end{align*}
Taking the residue at zero returns the coefficient at $z_1^{-1}\dots z_n^{-1}$. For the summand $A$ it is zero unless $\lambda = \rho(n) + 2 \mu$, and for $\lambda = \rho(n) + 2 \mu$ one has 
\[\res_{\zz=0} A = \sum_{\sigma \in \Sigma_n} (-1)^{sgn(\sigma)} t_1\dots t_n t_{\sigma(1)}^{2 \mu_1+2n-2}\dots t_{\sigma(n)}^{2\mu_n}.\]
For the summand $B$, taking the residue gives zero unless $\lambda = \rho(n-1) + 2 \mu$, and in the latter case the result is
\[\res_{\zz=0} B = \sum_{\sigma \in \Sigma_n} (-1)^{sgn(\sigma)} t_{\sigma(1)}^{2 \mu_1+2n-2}\dots t_{\sigma(n)}^{2\mu_n}.\]

Finally, recognizing the sums over permutations as determinants, one gets
\begin{align*} 
\pi_{*}^{+} s_\lambda(\RR^\vee) &= C \cdot \res_{\zz =0} (A+B)\\
&=\frac{2^{n-1} }{\prod_{i<j}(t_i^2 - t_j^2)} \cdot \res_{\zz =0} (A+B)\\
&=2^{n-1} (-1)^{n \choose 2}\begin{cases}
  s_{\mu}(t_1^2, \dots, t_n^2)  &  \lambda = 2 \mu + \rho(n-1) \\
  t_1\dots t_n s_{\mu}(t_1^2, \dots, t_n^2)  &  \lambda = 2 \mu + \rho(n)\\
  0 &  \text{otherwise}
\end{cases}.
\end{align*} 

For the connected component $\OG{n}{2n}^{-}$, one immediately has
\begin{align*} 
\pi_{*}^{-} s_\lambda(\RR^\vee) &= \frac{C}{\prod_{i<j}(t_j^2 - t_i^2)} \cdot \res_{\zz =0} (A-B)\\
&=2^{n-1}  \begin{cases}
  s_{\mu}(t_1^2, \dots, t_n^2)  &  \lambda = 2 \mu + \rho(n-1) \\
  -t_1\dots t_n s_{\mu}(t_1^2, \dots, t_n^2)  &  \lambda = 2 \mu + \rho(n)\\
  0 &  \text{otherwise}
\end{cases}.
\end{align*} 
\end{proof}

\section{K--theory} The residue formulas for the push-forward in K--theory of flag varieties were given in \cite{AlRi, RiSz}. Additional formulas for a range of homogeneous spaces were derived in \cite{wz}.
Studying the K--theory push-forward formula for the even orthogonal Grassmannian is somewhat more challenging. The primary reason is the absence of symmetry. The Euler class in K--theory\footnote{For a line bundle $L$ the K--theoretic  Euler class is equal to $eu^K(L)=1-[L^\vee]$.} is not symmetric, i.e. in general
$$eu^K(E)\neq (-1)^{\rk(E)}eu^K(E^\vee)\,.$$
As in the homological case the push-forwards from the components of $\OG{n}{2n}$ do not coincide.

\begin{example}
$$\pi^{+}_!(\wedge^2\RR^\vee)=\frac{t_1 t_2}{1 - t_1^{-1} t_2^{-1}} + \frac{t_1^{-1} t_2^{-1}}{1 - t_1 t_2}=
1 + t_1 t_2 + t_1^{-1} t_2^{-1}\,,$$
while
$$\pi^{-}_!(\wedge^2\RR^\vee)=\frac{t_1^{-1}t_2}{1 - t_1 t_2^{-1}} + \frac{t_1 t_2^{-1}}{1 - t_1^{-1}t_2}=
1+t_1t_2^{-1} + t_2^{-1}t_1\,.$$
\end{example}

For a partition $\lambda=(\lambda_1\geq \lambda_2\geq\dots \geq \lambda_n)$, $\lambda_n\geq0$, let $S_\lambda(\RR^\vee)$ be the Schur functor applied to the dual tautological bundle. The result $\pi^{+}_!(S_\lambda(\RR^\vee))$ can be easily described. Consider the fibration $$\nu:\sfrac{\SO_{2n}}{\BB_-}\longrightarrow \sfrac{\SO_{2n}}{\PP}\,,$$
where $\BB_-$ is the group of lower-triangular matrices from $\SO_{2n}$. The fiber $\sfrac{\PP}{\BB_-}$ is isomorphic to the  (classical) full flag variety. Moreover, $S_\lambda(\RR^\vee)$ is equal to the push-forward
$\nu^K_!(L_\lambda)$, where $L_\lambda=\PP\times_{\BB_-}\CC_\lambda$ is the line bundle defined by the character $\lambda$. This is just the Borel-Weil-Bott Theorem applied fiberwise. Thus
$$\pi^{+}_!(S_\lambda(\RR^\vee))=(\pi^+\circ\nu)_!(L_\lambda)=\chi(V_\lambda)$$
is the character of the highest weight representation $V_\lambda$ of $\SO_{2n}$. 
One can check that if $\lambda_n=0$ then $\pi^{+}_!(S_\lambda(\RR^\vee))=\pi^{-}_!(S_\lambda(\RR^\vee))$. 
For the remaining weights the representations $\pi^{+}_!(S_\lambda(\RR^\vee))$ and $\pi^{-}_!(S_\lambda(\RR^\vee))$ differ by the conjugation induced by an element of $\Oo_{2n}\setminus \SO_{2n}$.
This is the K--theoretic counterpart of Theorem \ref{thm:schur}. The same argument applies to generalized partitions satisfying  $\lambda_{n-1}+\lambda_n\geq 0$ instead of $\lambda_n\geq 0$, which parametrize all highest weights for $\SO_{2n}$.
\\

It seems that there is no easy analogue of the residue formula generalizing Theorems \ref{thm:push} and \ref{cor:push}.

\bibliographystyle{alpha}
\bibliography{biblOG}

\begin{thebibliography}{Zie18b}

\bibitem[AB84]{ab}
M.~F. Atiyah and R.~Bott.
\newblock The moment map and equivariant cohomology.
\newblock {\em Topology}, 23(1):1--28, 1984.

\bibitem[AR18]{AlRi}
Justin Allman and Rich\'{a}rd Rim\'{a}nyi.
\newblock {$K$}-theoretic {P}ieri rule via iterated residues.
\newblock {\em S\'{e}m. Lothar. Combin.}, 80B:Art. 48, 12, 2018.

\bibitem[BV82]{bv}
N.~Berline and M.~Vergne.
\newblock Classes caract\'eristiques \'equivariantes. {Formules} de
  localisation en cohomologie \'equivariante.
\newblock {\em C. R. Acad. Sci. Paris}, 295:539--541, 1982.

\bibitem[DP19]{dp}
Lionel Darondeau and Piotr Pragacz.
\newblock Gysin maps, duality, and {S}chubert classes.
\newblock {\em Fund. Math.}, 244(2):191--208, 2019.

\bibitem[Mac15]{MacD}
I.~G. Macdonald.
\newblock {\em Symmetric functions and {H}all polynomials}.
\newblock Oxford Classic Texts in the Physical Sciences. The Clarendon Press,
  Oxford University Press, New York, second edition, 2015.
\newblock With contribution by A. V. Zelevinsky and a foreword by Richard
  Stanley.

\bibitem[PR97]{PraRat}
P.~Pragacz and J.~Ratajski.
\newblock Formulas for {L}agrangian and orthogonal degeneracy loci; {$\tilde
  Q$}-polynomial approach.
\newblock {\em Compositio Math.}, 107(1):11--87, 1997.

\bibitem[RS22]{RiSz}
Richárd Rimányi and András Szenes.
\newblock {Residues, Grothendieck Polynomials, and K-Theoretic Thom
  Polynomials}.
\newblock {\em International Mathematics Research Notices}, page rnac345, 12
  2022.

\bibitem[Tev05]{tevelev}
E.~A. Tevelev.
\newblock {\em Projective duality and homogeneous spaces}, volume 133 of {\em
  Encyclopaedia of Mathematical Sciences}.
\newblock Springer-Verlag, Berlin, 2005.
\newblock Invariant Theory and Algebraic Transformation Groups, IV.

\bibitem[WZ19]{wz}
Andrzej Weber and Magdalena Zielenkiewicz.
\newblock Residues formulas for the push-forward in {K}-theory, the case of
  {${\rm G}_2/P$}.
\newblock {\em J. Algebraic Combin.}, 49(3):361--380, 2019.

\bibitem[Zie14]{mz0}
Magdalena Zielenkiewicz.
\newblock Integration over homogeneous spaces for classical {L}ie groups using
  iterated residues at infinity.
\newblock {\em Cent. Eur. J. Math.}, 12(4):574--583, 2014.

\bibitem[Zie18a]{mz1}
Magdalena Zielenkiewicz.
\newblock Pushing-forward {S}chur classes using iterated residues at infinity.
\newblock In {\em Schubert varieties, equivariant cohomology and characteristic
  classes---{IMPANGA} 15}, EMS Ser. Congr. Rep., pages 331--345. Eur. Math.
  Soc., Z\"{u}rich, 2018.

\bibitem[Zie18b]{mz}
Magdalena Zielenkiewicz.
\newblock Residue formulas for push-forwards in equivariant cohomology---a
  symplectic approach.
\newblock {\em J. Symplectic Geom.}, 16(5):1455--1480, 2018.

\end{thebibliography}

\end{document}